\documentclass{amsart}

\usepackage{amsmath}

\newcommand{\ep}{\varepsilon}

\renewcommand{\le}{\leq}
\renewcommand{\ge}{\geq}
\newcommand{\BV}{\operatorname{BV}}
\newcommand{\1}{{\bf 1}}

\newcommand{\Tan}{{\rm Tan}}

\newcommand{\restrict}{\llcorner}

\newcommand{\UPR}{{\mathcal U}_{\rm{PR}}}

\DeclareMathOperator{\INt}{int}
\DeclareMathOperator{\Span}{span}
\DeclareMathOperator{\Div}{div}
\DeclareMathOperator{\Star}{star}

\newtheorem{theorem}{Theorem}[section]
\newtheorem{lemma}[theorem]{Lemma}

\newtheorem{proposition}[theorem]{Proposition}
\newtheorem{corollary}[theorem]{Corollary}
\newtheorem{example}[theorem]{Example}
\newtheorem{remark}[theorem]{Remark}

\newcommand{\bbE}{{\mathbb E}}

\newcommand{\bbN}{{\mathbb N}}

\newcommand{\R}{{\mathbb R}}

\newcommand{\cF}{{\mathcal F}}
\newcommand{\cH}{{\mathcal H}}

\newcommand{\cB}{{\mathcal B}}

\newcommand{\cR}{{\mathcal R}}
\newcommand{\cM}{{\mathcal M}}

\newcommand{\conv}{{\mathrm{conv}}}

\newcommand{\ova}{\overline{P}}

\newcommand{\loc}{{\rm loc}}

\begin{document}

\title{Dilation volumes of sets of finite perimeter} 
\author{Markus Kiderlen} 
\address{Department of Mathematical Sciences, University of Aarhus, Ny
  Munkegade 118, DK-8000 Aarhus C, Denmark}
\email{kiderlen@math.au.dk}
\author{Jan Rataj}
\address{Charles University, Faculty of Mathematics and Physics, Sokolovsk\'a 83, 18675 Praha 8, Czech Republic}
\email{rataj@karlin.mff.cuni.cz}
\thanks{The first author was supported by the Villum Foundation in the
framework of the VKR Centre of Excellence {\bf CSGB} 
(Centre for Stochastic Geometry and Advanced Bioimaging). The second author acknowledges support from the Czech Science Foundation, project No.\ P201/15-08218S}

\begin{abstract}
This paper analyzes the first order behavior (that is, the right sided
derivative) of the volume of the
dilation $A\oplus tQ$ 
as $t$ converges to zero. Here $A$ and $Q$ are subsets of 
$n$-dimensional Euclidean space, $A$ has  finite perimeter and $Q$ is 
finite. If $Q$  consists of two points only, $x$ and $x+u$, say,
this derivative coincides up to sign with the directional derivative of the
covariogram of $A$ in direction $u$. By  known results for the
covariogram,   this
derivative can therefore be expressed by the cosine transform of the surface
area measure of $A$. We extend this result to finite sets $Q$
 and use it to determine the derivative of the 
contact distribution function with finite structuring element of a stationary random set at
zero.
 The proofs are based on approximation 
 of the characteristic function of $A$ by smooth functions of
bounded variation and showing  corresponding  formulas for them. 
\end{abstract}

\keywords{bounded variation; contact distribution function; dilation volume; directional variation; sets of finite perimeter; stationary random set; surface area measure} 

\subjclass[2000]{26B30,28A75,60D05} 

\maketitle

\section{Introduction}

Assume that $A\subset\R^n$ has regular boundary in the sense that the $(n-1)$-dimensional Hausdorff measure of its boundary, $\cH^{n-1}(\partial A)$, is finite and that for $\cH^{n-1}$ almost all $a\in\partial A$, there exists a unique outer unit normal vector $\nu_A(a)\in S^{n-1}$. This is the case e.g.\ if $A$ is a topologically regular convex or polyconvex set, $n$-dimensional compact Lipschitz manifold with boundary or a ``full-dimensional $\UPR$ set'' (\cite{rata04}). Then, the surface area measure of $A$ is defined naturally as
\[
S_{n-1}(A;\cdot)=\cH^{n-1}\{a\in\partial A:\, \nu_A(a)\in\cdot\}.
\]
The surface area measure is an important quantity in stochastic geometry and its estimation is a frequent task. Various integral formulas are used in this context. It is well-known that the intersection density of $\partial A$ with lines of direction $u\in S^{n-1}$ is
$$\int_{S^{n-1}} |u\cdot v|\, S_{n-1}(A;dv),$$
and that these integrals (called cosine transform) determine only the
symmetrized form of the surface area measure,
$S_{n-1}(A;\cdot)+S_{n-1}(-A;\cdot)$. The cosine transform appears
also in the directional derivative of the 
covariogram of $A$,
\[
C(A,y)=\lambda_n(A\cap (A+y)),\quad y\in\R^n,
\]
($\lambda_n$ is Lebesgue-measure in $\R^n$), as 
\begin{equation} \label{covariogram}
\lim_{r\to 0+} {\frac{C(A,ru)-C(A,0)}{r}}=-\frac 12 \int_{S^{n-1}} |u\cdot v|\, S_{n-1}(A;dv),
\end{equation}
when $u\in S^{n-1}$ and $A$ has finite volume. This 
was shown by Matheron \cite{mat:65} for convex bodies and  extended
considerably by Galerne \cite{Gal:10}. 

Note that the covariogram can be expressed by means of dilation volumes with two-point test sets, namely
$$C(A,y)=2\lambda_n(A)-\lambda_n(A\oplus\{ 0,y\}).$$
A natural extension is to consider the dilation volume
$\lambda_n(A\oplus Q)$
for  a general compact test set $Q\subset\R^n$. 
Generalizing  \eqref{covariogram}, we have 
\begin{equation}  \label{ext-cov}
\lim_{r\to 0+}\frac{\lambda_n(A\oplus rQ)-\lambda_n(A)}{r}=\int_{S^{n-1}} h(Q,v)\, S_{n-1}(A;dv),
\end{equation}
where $h(Q,\cdot)$ is the support function of $\conv \,Q$. This 
 was shown in \cite[Corollary~2]{KR06} under the assumption that $A$
 is a compact gentle set. Besides a technical regularity
 condition this means  that for $\cH^{n-1}$-almost all  points 
 $a\in \partial A$ there are non-degenerate osculating balls containing $a$,
 one completely contained in  $A$ and the other in the closure of
 $A^C$. While the right hand side of \eqref{covariogram} (known for all $u$) 
determines only the
symmetrized form of the surface area measure, the right hand side of
 \eqref{ext-cov} determines $S_{n-1}(A;\cdot)$ itself, when the
 integrals are known for all sets $Q$ that are congruent to a fixed
 triangle having at least one angle that is an irrational multiple of
 $\pi$. This was shown by Schneider \cite{schn74b}; see also 
\cite[p.~283 and (5.1.18)]{sch93}. In particular, for the
 determination of $S_{n-1}(A;\cdot)$ it is enough to know the right
 hand side of \eqref{ext-cov} for all 
 three-point test sets $Q$; cf.~\cite{rata04} for a related result. 
 \medskip

Although the class of gentle sets is reasonably large (it contains for
 instance all topologically regular polyconvex sets) 
 this condition for the derivation of \eqref{ext-cov} seems
to be rather artificial and its purpose is to make the proofs 
work. A different approach is based on the theory of 
sets with finite perimeter which are, by definition, sets
$A\subset\R^n$ whose indicator  function $\1_A$ has
distributional derivative representable 
as a Radon measure $D\1_A$. (In other words, $\1_A$ has bounded variation.) 
The notion of sets with finite perimeter goes back to Caccioppoli
 \cite{cac:52} and 
De Giorgi \cite{gio:54,gio:55,gio:58}. 
We note that  (poly)convex sets, compact $\UPR$-sets as
 well as compact gentle sets, or compact Lipschitz domains are sets of finite perimeter, simply as any set whose boundary has
finite $\cH^{n-1}$-measure has also finite perimeter.

In the following we describe, how 
the notion of surface area measure can be extended to sets of finite
perimeter. 
The \emph{essential  boundary} $\partial^* A$ of a set $A$ is the set of
points in $\R^n$ that are neither Lebesgue density points of $A$ nor
of its complement. If $A$ is a set of finite perimeter, then the variation (scalar) measure $|D\1_A|$ can be written as a restriction of the $(n-1)$-dimensional Hausdorff measure $\cH^{n-1}$ in the form 
\begin{align} \label{D1_A_Interpretation}
|D\1_A|=\cH^{n-1}\restrict (\partial^* A),
\end{align}
 \cite[(3.63)]{AFP00},
and the perimeter $P(A)=|D\1_A|(\R^n)$ equals
$\cH^{n-1}(\partial^* A)$. In the case where 
 $A$ has  Lipschitz boundary, we have
$\partial A=\partial^* A$ and $P(A)$ coincides with the usual surface area of
$A$. 

For a general set $A$ with finite perimeter, the distributional derivative $D\1_A$ can be decomposed as
$$D\1_A=\Delta_{\1_A}\,|D\1_A|;$$
see \eqref{polar}, below. 
The density $\Delta_{\1_A}$ is an $S^{n-1}$-valued function defined
$\cH^{n-1}$-almost everywhere on $\partial^*A$ and can be interpreted
as a generalized inner unit normal vector field to $A$. (In fact there
exists a subset of $\partial^*A$ of full $\cH^{n-1}$ measure called
{\it reduced boundary} and a representative $\nu_A$ of $-\Delta_{1_A}$
defined there such that the half-space $\{y:\, y\cdot\nu_A(a)\leq 0\}$
coincides with the approximate tangent cone of $A$ at $a$ for any $a$
from the reduced boundary, see \cite[\S3.5]{AFP00}.) Thus, it is 
natural to define the {\it generalized surface area measure} of a set $A$ with finite perimeter as
\begin{equation}  \label{def_SAM}
S_{n-1}^*(A;\cdot)=\cH^{n-1}\{a\in\partial^*A:\, -\Delta_{\1_A}(a)\in\cdot\}.
\end{equation}
Clearly, $S_{n-1}^*(A;\cdot)$ coincides with $S_{n-1}(A;\cdot)$ if $A$ has Lipschitz boundary.
\medskip

Sets with finite perimeter have already appeared in the context of stochastic
geometry. 
Villa 
\cite{
Vil:10}
considered the (outer) Minkowski content 
and the spherical contact
distribution function of inhomogeneous Boolean models with grains that
have finite perimeter. The second author  considered  in \cite{Rataj:15} random sets of finite perimeter and  established, among other things, a Crofton formula for these. Galerne and Lachi\`eze-Rey \cite{Gal:LachRey:15} developed a theory of random measurable (not necessarily closed) sets and discussed the covariogram realizability problem in this framework. Their paper is based on an earlier one by
Galerne \cite{Gal:10}, who showed an extension of the formula \eqref{covariogram} for sets with finite volume and finite perimeter, namely 
\begin{equation} \label{covariogram_e}
\lim_{r\to 0+} {\frac{C(A,ru)-C(A,0)}{r}}=-\frac 12 \int_{S^{n-1}} |u\cdot v|\, S^*_{n-1}(A;dv),
\end{equation}
and applied it to random sets. 

Our main result is an analogous extension of \eqref{ext-cov} for the case of finite sets $Q$:

\begin{theorem}\label{thm:main}
Assume that $A\subset\R^n$ has finite perimeter. If
 $\emptyset \ne Q\subset\R^n$ is finite then 
\begin{align}\label{eq:main1}
\lim_{r\to 0+}\frac {\lambda_n((A\oplus rQ)\setminus
  A)}r=\int h(Q,v)^+ \,S_{n-1}^*(A;dv).
\end{align}
If, in addition, $A$ has bounded volume then also
\begin{align}\label{eq:main2}
\lim_{r\to 0+}\frac {\lambda_n(A\oplus rQ)-\lambda_n(A)}r=
  \int h(Q,v)\, S_{n-1}^*(A;dv).
\end{align}
\end{theorem}

We show in Example~\ref{ex:1} that the result is no longer true if we allow $Q$ to be infinite, even if $Q$ is countable and compact.

The case when $Q$ is an $n$-dimensional convex body was considered by Chambolle et al. \cite{CLL:14}. They showed that \eqref{eq:main2} is true whenever it holds for $Q=B(0,1)$ (which, however, need not be true). They also proved the convergence in \eqref{eq:main2} in a weaker sense ($\Gamma$-convergence) for any $n$-dimensional convex body $Q$. Related results for special sets $A$ can be found in \cite{LV:16}.

Extending or complementing corresponding
results in \cite{Vil:10} and \cite{Gal:10}, we conclude with an
 application of Theorem~\ref{thm:main} for the contact distribution of stationary random sets. 
 Recall that for a stationary random \emph{closed} set $Z\subset\R^n$ (in the sense of Matheron; see, e.g.~\cite{SW:08}) with volume fraction $\overline{p}=\Pr(0\in Z)$,  the contact distribution function of $Z$ with compact structuring element $Q\subset \R^n$ is defined by 
 \[
 H_Q(r)=\Pr (Z\cap rQ\neq\emptyset\mid 0\not\in Z),\quad r\geq 0.
 \]
 We will derive a formula for the one-sided derivative of $H_Q$ at zero when $Q$ is finite.
The framework of sets with finite perimeter seems to be particularly well-suited
for this problem, as the result does not require any of the usual integrability assumptions. In addition, it even holds for the more general class of 
random measurable sets (RAMS) introduced in  \cite{Gal:LachRey:15}.

A RAMS is a random element from the space of Lebesgue measurable subsets of $\R^n$ modulo differences of Lebesgue measure zero, with topology induced by the $L^1_{\loc}$ convergence of the indicator functions. This setting includes random closed sets in the sense of Matheron as a special case.
The definitions of the volume fraction $\overline{p}$ and the contact distribution function $H_Q$ can be extended to stationary RAMS $Z\subset \R^n$; see Section \ref{s:appl}.
We use the notion of {\it specific perimeter} $\ova(Z)$ of $Z$ given as the (constant) density of the variation measure $|D\1_Z|$ with respect to $\lambda_n$ (cf.\ \cite{Gal:10} where the notion `specific variation' is used, or \cite{Rataj:15}), and oriented rose of directions ${\mathcal R}^*$ given as the distribution of the outer normal $-\Delta_{\1_Z}(z)$ at a typical point $z\in\partial^*Z$ in case $\overline{P}(Z)<\infty$; see Section~\ref{s:appl} for exact definitions.

\begin{theorem}\label{thm:contact}
Let $Q\ne\emptyset$  be finite. If 
$Z$ is a stationary RAMS, then the right sided derivative
$H_Q'(0+)$ of $H_Q$ at $0$ satisfies 
\begin{align}\label{eq:contactDistStat}
  (1-\overline p)H_Q'(0+)=\ova(Z)\int_{S^{n-1}} h(-Q,v)^+{\cR}^*(dv)
\end{align}
when  $\ova(Z)<\infty$. If  $\ova(Z)$ is infinite, and 
$\conv (Q\cup\{0\})$ has interior points, 
$$(1-\overline p)H_Q'(0+)=\infty.$$ 

If $Z$ is stationary and isotropic, and $\ova(Z)\in [0,\infty]$, then 
\begin{align}\label{eq:isotrop}
 (1-\overline p)H_Q'(0+)=\frac{1}{2}b(\conv(Q\cup \{0\}))\ova(Z)
\end{align}
where $b(\cdot)$ is the mean width. 
\end{theorem}


We would like to stress that the methods of proofs are different from the classical approaches in stochastic geometry when dealing with sets with finite perimeter. Namely, we use typically
approximations of characteristic functions by smooth functions of
bounded variation, show related formulas for them, and apply
continuity arguments to obtain the desired results. This means that we
have to define functionals to be dealt with not only for sets but also
for functions. 

\medskip

The paper is organized as follows. In Section~\ref{sec:prelim} we recall the
usual and directional variation of a function $f$, 
discuss basic properties, and define sets of finite perimeter. 
The notion of the variation $V^Q(f)$ of $f$ with respect to a compact set $Q$ is introduced and discussed in Section~\ref{sec:Vq}. This is a special case of anisotropic variation with respect to a Finsler metric, see \cite{AB:94}. In particular, $V^{-Q}(\1_A)$ coincides with
 the right hand side of \eqref{eq:main2} when $0\in Q$. 
Section~\ref{s:dv} is devoted to the proof of the main result, 
Theorem~\ref{thm:main}. While one equality (Proposition~\ref{PG}) is obtained by standard methods (similarly as the same inequality for $n$-dimensinal convex bodies in \cite{CLL:14}), the other inequality (Corollary~\ref{cor:limsup}) is more difficult. The above mentioned application to
random sets and the proof of Theorem~\ref{thm:contact} is described in Section~\ref{s:appl}.

\section{Preliminaries}\label{sec:prelim}
We present here some definitions and properties of functions of bounded variation and sets with finite perimeter. As reference we use mostly the book \cite{AFP00}.

Let $\Omega$ be a nonempty open subset of $\R^n$ and $0\ne u\in\R^n$. 
We write  $L_{\mathrm{loc}}^1(\Omega)$ 
	for the space of all functions on $\Omega$ that are locally Lebesgue-integrable.
The {\it distributional directional derivative} of a function $f\in L_{\mathrm{loc}}^1(\Omega)$ in direction $u$ is the linear functional
\begin{align}\label{eq:du}
D_uf: \phi\mapsto -\int_{\Omega}\frac{\partial\phi}{\partial u}(x)f(x)\, dx,\quad \phi\in C^\infty_c(\Omega).
\end{align}
Here $\frac{\partial\phi}{\partial u}(x)$ is the classical
directional derivative of a smooth function, $dx$  denotes the
integration w.r.t.\ Lebesgue measure and $C^\infty_c(\Omega)$ stands
for the space of infinitely differentiable functions on $\Omega$ with
compact support. 
We define the {\it directional variation} of $f\in L_{\mathrm{loc}}^1(\Omega)$ in the direction $u\in S^{n-1}$ as
$$V_u(f,\Omega):=\sup\left\{D_uf(\phi):\, \phi\in C^\infty_c(\Omega),\, \|\phi\|_\infty\leq 1\right\}.$$
If the last expression is finite and $f\in L^1(\Omega)$, we say that $f$ has {\it finite
  directional variation} (in $\Omega$ and) in direction $u$. We
denote by $\BV_u(\Omega)$ the space of all such functions. 
Note that, by the Riesz representation theorem, $f\in\BV_u(\Omega)$ if and
only if the
distributional directional derivative $D_uf$ can be represented as a 
finite Radon measure on $\Omega$. In this case we have 
$V_u(f,\Omega)=|D_uf|(\Omega)$, where $|\mu|$ denotes the variation
measure of the (real- or vector-valued) Radon measure $\mu$ given by 
\[
|\mu|(A)=\sup\left\{\sum_{h=1}^\infty |\mu(E_h)|: (E_1,E_2,\ldots) \text{
	forms a Borel partition of } A\right\}
\]
for any Borel set $A\subset \Omega$.

The {\it variation} of a function $f\in L^1_{\mathrm{loc}}(\Omega)$ is defined as
\[
  V(f,\Omega):=\sup\left\{\int_{\Omega}f(x)\Div\varphi(x)\, dx:\,
  \varphi\in C^\infty_c(\Omega,\R^n),\, \| |\varphi|\|_\infty\leq
  1\right\}.
\]

Here, $C^\infty_c(\Omega,\R^n)$ is the vector space of $\R^n$-valued 
infinitely differentiable functions on $\Omega$ with
compact support, and
$\||\varphi|\|_\infty$ is the $L^\infty$-norm of the 
Euclidean norm $|\varphi|=\sqrt{\varphi_1^2+\cdots +\varphi_n^2}$ of 
 $\varphi=(\varphi_1,\ldots,\varphi_n)$.
If  $V(f,\Omega)$  is finite and $f\in L^1(\Omega)$, we say that $f$ has {\it bounded
  variation} {in $\Omega$}. The vector space of all functions
 of bounded variation is denoted by
$\BV(\Omega)$.  
Functions  $f\in L_{\mathrm{loc}}^1(\Omega)$ with bounded variation in any relatively compact open subset of $\Omega$ are said to have \emph{locally} bounded
		variation {in $\Omega$}.  We have $f\in\BV(\Omega)$ if and only if
$f\in\BV_u(\Omega)$ for all $u\in S^{n-1}$ and then,
\begin{align}\label{eq:rotatMittel}
V(f,\Omega)=
(2\kappa_{n-1})^{-1}\int_{S^{n-1}}V_u(f,\Omega)\,
\cH^{n-1}(du),
\end{align}
cf.\ \cite{Gal:10}. Here and in the following   $\cH^k$ denotes the 
$k$-dimensional Hausdorff measure  in $\R^n$, and
$\kappa_k$ is the $k$-dimensional volume of the Euclidean unit ball  in
$\R^k$.

If $f\in\BV(\Omega)$ then there exists a finite  $\R^n$-valued Radon
measure $Df$ on $\Omega$ such that $Df(A)\cdot u=D_uf(A)$ for all
Borel-sets $A\subset\Omega$, and $u\ne 0$; 
$Df$ represents the {\it distributional derivative} of $f$, cf.\ \cite[\S3.1]{AFP00}.
The variation of $f$ is the total variation of $Df$: 
\begin{align} \label{eq:VisVariation} V(f,\Omega)=|Df|(\Omega).
\end{align} 

Let 
\begin{equation}  \label{polar}
Df=\Delta_f\, |Df|
\end{equation}
be the {\it polar decomposition} of $Df$, i.e., $\Delta_f\in
L^1(\Omega,|Df|)$ taking values in $S^{n-1}$ is the Radon-Nikod\'ym
density of $Df$ w.r.t.\ $|Df|$ (cf.\ \cite[Corollary 1.29]{AFP00}). 
Note that if $f\in\BV(\Omega)$ and $u\ne 0$ then $V_u(f,\Omega)$ 
can be written in the form
$$V_u(f,\Omega)=\int_\Omega |u\cdot\Delta_f(x)|\, |Df|(dx).$$

Note also that if $f\in C^1(\Omega)\cap \BV(\Omega)$ then $Df(dx)=\nabla f(x)\, dx$, $|Df|(dx)=|\nabla f(x)|\, dx$, and
$$\Delta_f(x):=\begin{cases}
\frac{\nabla f(x)}{|\nabla f(x)|},&\nabla f(x)\neq 0,\\
0,&\text{otherwise},
\end{cases}$$
is a version of $\Delta_f$, where $\nabla f(x)$ denotes the gradient of $f$ at $x$.

Let $(f_j)$ be a sequence of functions in $\BV(\Omega)$ and let 
$f\in\BV(\Omega)$. Following \cite[3.14]{AFP00}, 
we say that $(f_j)$ {\it converges strictly} to $f$ if 
$f_j\to f$ in $L^1(\Omega)$ and, additionally, 
$V(f_j,\Omega)\to V(f,\Omega)$. 
As a basic example, consider any function $f\in\BV(\Omega)$ and a
sequence of $C^\infty$ mollifiers $\rho_j$ (i.e.,
$\rho_j(y)=j^n\rho(jy)$ with a {nonnegative} function $\rho\in C^\infty_c$ fulfilling
$\int_{\R^n}\rho dx =1$). Then, the convolutions $f*\rho_j$ ({\it mollifications of $f$}) belong to $C^\infty(\Omega')$ and $f*\rho_{{j}} \to f$ strictly in a slightly 
``shrunk'' open set $\Omega'=\{x\in \Omega:
	\text{dist}(x,\partial \Omega)>\varepsilon\}$ 
(cf.\ \cite[\S2.1,~3.1]{AFP00}).
That the set $\Omega$ has to be replaced by a smaller one can be avoided by mollifying $f \varphi_h$, where $(\varphi_h)$ 	is a smooth partition of unity in $\Omega$ relative to a locally finite covering $(\Omega_h)$     with open, relative compact sets. The corresponding result can be found in \cite[Theorem~5.3.3]{Ziemer} and implies the third statement in the following collection of well-known basic properties of the variation.

\begin{proposition}[Basic properties of the variation]  \label{prop:Var}  \mbox{} \\  \vspace{-2ex}
	\begin{itemize}
		\item[(a)] For $f\in \BV(\Omega)\cap C^1(\Omega)$, 
		\begin{align*}  V(f,\Omega)=\int_{\Omega} |\nabla f|dx.
		\end{align*} 
		\item[(b)] If  $f_j\to f$ in $L^1(\Omega)$ then $ V(f,\Omega)\le \liminf_{j\to\infty}V(f_j,\Omega)$.
		\item[(c)] For $f\in \BV(\Omega)$, there is a sequence  of functions $(f_j)$ in $C^\infty(\Omega)\cap \BV(\Omega)$ converging strictly to $f$.
	\end{itemize}
\end{proposition}

The following lemma states that the positive and negative parts of  $D_uf$, $u\ne 0$,  have the same total mass when $f\in \BV(\Omega)$ and $\Omega=\R^n$. This is not necessarily true when $\Omega\ne \R^n$. For instance, $f(x)=x$ on $\Omega=(0,1)$ satisfies 
$(D_1f)^+(\Omega)=1$, but $(D_1f)^-(\Omega)=0$. 
\begin{lemma}\label{lem:lin}
	For $f\in \BV$ we have $Df(\R^n)=0$. In particular, 
\begin{align}\label{centerCond}
D_uf(\R^n)=\int_{\R^n} \left(u\cdot \Delta_f(x)\right) |Df|(dx)= 0 
\end{align} 
for all $u\ne 0$. 
\end{lemma}

\begin{proof} Fix $f\in \BV$ and put $\phi_m=\1_{B(0,m)}*\rho$,  where $0\le \rho\in C^\infty$ is a mollifier with support in $B(0,1)$.  Clearly, $\nabla \phi_m$ is zero outside the annulus $R_m=B(0,m+1)\setminus B(0,m-1)$, 
		and  $\|\frac{\partial \phi_m}{\partial x_i}\|_\infty<\kappa_n \||\nabla \rho|\|_\infty$, so
		\[
		\left|	\int_{\R^n} \phi_m (Df)_i(dx)\right|=\left|-\int_{\R^n}\frac{\partial \phi_m}{\partial x_i}f(x)dx\right|\le \kappa_n \||\nabla \rho|\|_\infty\int_{R_m}|f|dx.
		\] for all
		$i\in\{1,\ldots,n\}$. As $f\in L^1$, the right hand side converges to $0$. The left hand side converges to $\left|	\int_{\R^n} (Df)_i(dx)\right|$, as 
		$\phi_m$ is an increasing sequence with pointwise limit $1$, and $(Df)_i$ is a finite Radon measure. We conclude 
		$Df(\R^n)=0$ and 
		\begin{align*}
		\int_{\R^n} \left(u\cdot \Delta_f(x)\right) |Df|(dx)=
		u\cdot \int_{\R^n} Df(dx)= 0,
		\end{align*} as claimed.
\end{proof}

\medskip

We shall work with the following generalization of directional
variations. Let $L$ be a linear subspace of $\R^n$ of dimension $k\in \{1,\ldots,n\}$. If $C^\infty_c(\Omega,L)$ denotes the vector space of all functions in $C^\infty_c(\Omega,\R^n)$ with values in $L$, 
we may define the \emph{$L$-variation} in $\Omega$ of $f\in L_{ \mathrm{loc}}^1(\Omega)$  as
\[
V_L(f,\Omega):=\sup\left\{\int_{\Omega}f(x)\sum_{i=1}^k\frac{\partial
   ( \varphi\cdot u_i)}{\partial u_i}(x)\, dx:\, \varphi\in
  C^\infty_c(\Omega,L),\, \| |\varphi|\|_\infty\leq 1\right\},
\]
where  $\{u_1,\ldots,u_k\}$ is an orthonormal basis of $L$. This definition does not depend on the choice of the orthonormal basis.

Clearly, when $f\in L^1(\Omega)$, $V_L(f,\Omega)<\infty$ if and only if  $V_u(f,\Omega)<\infty$
for all unit vectors $u\in L$, and in this case, we say that  $f$ has \emph{finite directional variation in $L$}, writing  $f\in \BV_L(\Omega)$.
We have $V_{\R^n}(f,\Omega)=V(f,\Omega)$, and 
$V_{\Span\{u\}}(f,\Omega)=V_u(f,\Omega)$ when $u\in S^{n-1}$.  
If $L\subseteq L'$ are two subspaces then $V_L(f,\Omega)\leq V_{L'}(f,\Omega)$.

\begin{proposition}[Basic properties of the directional variation]  \label{propDirVar}
The following assertions hold for a linear subspace $\{0\}\ne L\subset \R^n$. 
\begin{enumerate}
	\item[(a)] For $f\in \BV(\Omega)$ we have 
         \begin{equation}  \label{E-VL}
               V_L(f,\Omega)=|p_L (Df)|(\Omega)=\int_{\Omega}|p_L\Delta_f(x)|\, |Df|(dx), 
         \end{equation}
         where $p_L$ denotes the orthogonal projection on $L$. If, in addition, $f\in C^1(\Omega)$, 
         \[V_L(f,\Omega)=\int_{\Omega} |p_L \nabla f|dx.
         \]
\item[(b)] If $f_j\to f$ in $L^1(\Omega)$ then $V_L(f,\Omega)\leq\liminf_{j\to\infty} V_L(f_j,\Omega)$.
\item[(c)] If $(f_j)$ is a sequence converging strictly to $f\in\BV(\Omega)$, then  $V_L(f_j,\Omega)\to V_L(f,\Omega)$ as $j\to \infty$. 
\end{enumerate}
\end{proposition}
\begin{proof}
	The first two statements  generalize slightly \cite[Proposition~3.6]{AFP00} and we can skip the proof as it is quite obvious. To show (c) let 
	$(f_j)$ be a sequence converging strictly to $f\in\BV(\Omega)$.
	 By \cite[Proposition~3.13]{AFP00} the measures $Df_j$ converge weakly to $Df$ in $\Omega$ and their total variations converge to $|Df|(\Omega)$. The claim now follows from a special case of the Reshetnyak continuity theorem, Lemma \ref{L-Resh}, below, which is quoted here from the literature for easy reference.
\end{proof}

\begin{lemma}[{\cite [Proposition~2.39]{AFP00}}]  \label{L-Resh}
	Let $\mu_0,\mu_1,\ldots$ be finite vector-valued Radon measures on an open set $\Omega\subset \R^n$, such that $\mu_j$ converges weakly to $\mu_0$ in $\Omega$ and $|\mu_j|(\Omega)\to |\mu_0|(\Omega)$ as $j\to \infty$. Then
	$$\int_{\Omega}h(g_j(x))\, d|\mu_j|(x)\to \int_{\Omega}h(g_0(x))\, d|\mu_0|(x), \quad j\to \infty,$$
	for every  continuous and bounded function $h:\Omega\to\R$, where $g_j$ is the Radon-Nikod\'ym
	density of $\mu_j$ with respect to $|\mu_j|$. 
\end{lemma}
\medskip

The {\it perimeter} of a measurable set $A\subseteq\R^n$ in an open set $\Omega$ is defined as
$$P(A,\Omega)=V(\1_A,\Omega).$$
If the last quantity is finite, we say that $A$ has {\it finite
  perimeter in $\Omega$}. Sets $A$ with 
$P(A,\R^n)<\infty$ are simply called sets of {\it finite   perimeter}. This class is closed under set complement operation: a Borel set $A$ has finite perimeter if and only if its complement has finite perimeter.
In all the above notions, we  
skip from now on the argument $\Omega$ if $\Omega=\R^n$. If $A$ has finite volume, $\1_A$ is in $L^1(\R^n)$ and thus $A$ has finite perimeter if and only if $\1_A\in \BV$. 

If the perimeter
of a set $A$ is finite, it is the variation  of $D(\1_A)$ on $\Omega$. This variation  measure  can  in turn be expressed by means of the $(n-1)$-dimensional Hausdorff
measure. To do so, let the \emph{reduced boundary} $\cF A$  be the set of all points $x\in \Omega$ in the support of  $|D(\1_A)|$ such that the limit 
\[
\nu_A(x)=-\lim_{\rho\to 0_+} \frac{D(\1_A)(B(0,\rho))}{|D(\1_A)|(B(0,\rho))}
\]
exists in $\R^{n}$ and is a unit vector. Here and in the following, $B(x,\rho)$ denotes the Euclidean ball with radius $\rho\ge 0$ centered at $x\in\R^n$.
 The negative sign in this definition is included here, so that the  function   $\nu_A:\cF A\to S^{n-1}$ can be interpreted as \emph{generalized outer normal to  $A$}. By the Besicovitch derivation theorem \cite[Theorem 2.22]{AFP00}, $|D(\1_A)|$  is concentrated 
on $\cF A$, and $D(\1_A)=-\nu_A |D(\1_A)|$. 
A comparison with the polar decomposition \eqref{polar} yields $-\nu_A(x)=\Delta_{\1_A}(x)$ for $|D(\1_A)|$-almost every $x$. 
De Giorgi has shown that $\cF A$ is countably $(n-1)$-rectifiable and  $|D(\1_A)|=\cH^{n-1}\restrict \cF A$, see, for instance \cite[Theorem 3.59]{AFP00}.

If $A\subset\R^n$ has finite perimeter and $u\in\R^n$, let 
$\cF_{u+}A$, $\cF_{u-}A$ denote the set of all points $x\in\cF A$ such that $\nu_A(x)\cdot u$ is positive or negative, respectively. When $u\not=0$, these sets are connected to the positive and negative parts of the measure $D_u\1_A$ as follows:
\begin{align}
(D_u\1_A)^+(B)&=\int_{B\cap\cF_{u-}A}|u\cdot\nu_A(x)|\, \cH^{n-1}(dx),\nonumber\\
(D_u\1_A)^-(B)&=\int_{B\cap\cF_{u+}A}|u\cdot\nu_A(x)|\, \cH^{n-1}(dx),  \label{Du-}
\end{align}
where $B$ is any bounded Borel subset of $\R^n$.

It is sometimes convenient to replace $\cF A$ with larger sets, that are easier to handle. Let $\partial^*A=\R^n\setminus (A^0\cup A^1)$ be the 
\emph{essential boundary} of $A$, where 
\begin{align}\label{def:Lebesgue}
   A^t:=\left\{x\in \R^n: \lim_{r\to 0+} \frac{\lambda_n(A\cap
       B(x,r))}{\lambda_n(B(x,r))}=t\right\}
\end{align}
is the set of all points with \emph{Lebesgue  density $t\in [0,1]$}. 
Then we have $\cF A\subset \partial^* A$ (\cite[Theorem~3.61]{AFP00}). If $A$ is a set of finite perimeter in $\Omega$,
it can be shown that  $\cH^{n-1}(\Omega\cap \partial^{*}A\setminus \cF A)=0$, see \cite[Theorem 3.61]{AFP00}, and  thus we have 
\begin{equation}\label{obsolete}
    |D(\1_A)|=\cH^{n-1}\restrict \cF A=\cH^{n-1}\restrict \partial^* A
\end{equation}
on $\Omega$, and, in particular,
\begin{align}\label{eq:pa}
P(A,\Omega)=\cH^{n-1}(\cF A\cap \Omega)=
\cH^{n-1}( \partial^*  A\cap \Omega).
\end{align}
When $\Omega=\R^n$, the \emph{generalized surface area measure} of $A$, as defined in the introduction, can therefore also be written as 
\begin{equation}\label{eq:surfArea}
S_{n-1}^*(A;\cdot)=\cH^{n-1}\left(\{a\in\cF A:\nu_A(a)\in\cdot\}\right).
\end{equation}

\begin{remark} \rm
As functions with bounded variation, sets with finite perimeter are considered not as individual sets in $\R^n$, but as equivalence classes $\1_A\in L^1$. Thus, two sets with finite perimeter are considered as identical if the Lebesgue measure of their symmetric difference vanishes.
\end{remark} 

\section{The variation with respect to a compact set} \label{sec:Vq}
The \emph{support  function} $h(Q,\cdot)$ 
of a non-empty compact set  $Q$ in $\R^n$  is defined as the 
(usual) support function of its convex hull $\conv Q$. Explicitly, we have
\[
  h(Q,u)=\max\{ u\cdot x: x\in Q\},\qquad u\in S^{n-1}. 
\]
If $x^+=\max\{x,0\}$ denotes  the positive part of
$x\in \R$,  $h(Q\cup \{0\},\cdot)=h(Q,\cdot)^+$.  
Properties and applications of  the support function of convex sets
can be found in \cite{sch93}. We only mention here that the
\emph{mean width} $b(K)$ 
of a non-empty compact convex  set $K\subset \R^n$
can be defined using its support function:  
\[
   b(K)=\frac{2}{n\kappa_n}\int_{S^{n-1}} h(K,u) du.
\]
For an open set $\Omega\subset \R^n$ and $f\in \BV(\Omega)$
 with polar decomposition \eqref{polar}, we define a functional 
\[
V^Q(f, \Omega)=\int_{\Omega}h(Q,\Delta_f(x))^+ \,|Df|(dx)
\]
and call it the \emph{variation of $f$ with respect to $Q$ {in $\Omega$}}. As $V^Q(f{, \Omega})=V^{\conv (Q\cup\{0\})}(f{, \Omega})$, 
this variation depends on $Q$ only through the convex hull of $Q\cup\{0\}$.
We follow our usual convention and write $V^Q(f)=V^Q(f, \R^n)$. If this definition is applied to the  indicator function of a set $A\subset\R^n$ of finite perimeter with $\Omega=\R^n$, \eqref{obsolete} and \eqref{eq:surfArea} give 
\begin{align}\label{mixedvol}
	V^Q(\1_A)=\int_{S^{n-1}} h(-Q\cup\{0\},u)S_{n-1}^*(A;du).
\end{align}
If $A$ is a convex body, $V^Q(\1_A)=nV(-Q\cup\{0\},A,\ldots,A)$ is a mixed  volume, so $V^Q(\1_A)$ generalizes certain mixed volumes to sets of finite perimeter. 
If $\conv Q$ is symmetric w.r.t.\ the origin then $V^Q(f)$ is a special case of the generalized (anisotropic) variation defined in \cite{AB:94}. Indeed, we have $V^Q(f)=|Df|_{\phi}(\R^n)$ with Finsler metric $\phi(x,v)=h_{Q\cup\{0\}}(v)$, $x\in\R^n$, $v\in\R^n\setminus\{0\}$, in the sense of \cite[Definition~3.1]{AB:94}.

Let $B_L=B(0,1)\cap L$ be the unit ball in $L$.
One motivation to call $V^Q(f)$ a ``variation'' comes from the fact that 
\begin{align}\label{eqB_L}
     V^{B_L}(f,\Omega)=V_L(f,\Omega),
\end{align}
which follows directly from the definitions as $h(B_L,\cdot)=|p_L|$.  
In particular, we have $V^{\{-u,u\}}(f,\Omega)=V_u(f,\Omega)$ whenever $u\in S^{n-1}$. 
Another motivation is that averaging $Q$-variations gives the usual variation, that is,
\[
\int_{SO(n)} V^{\vartheta Q}(f,\Omega)\, d\vartheta=c_Q\, V(f,\Omega).
\]
where $c_Q=(1/2)b\left(\conv (Q\cup\{0\})\right)$. This follows directly from the definitions and an application of Fubini's theorem. 
We now summarize connections and basic inequalities
between the variation with respect to $Q$ and the $L$-variations when $\Omega=\R^n$. 

\begin{lemma}[Ordinary variation and variation with respect to $Q$]
\label{lem:VqV}
 Let  $f\in \BV$ and a non-empty compact set 
 $Q\subset \R^n$ be given. Then the following statements hold.
\begin{enumerate}
\item[\rm (a)] For $u\in S^{n-1}$ we have
\[
   2 V^{\{u\}}(f)=V_u(f). 
\]
\item[\rm (b)]  $V^Q(f)\ge s V_L(f)$, where $L=\Span Q$ and $s$ is the
  relative inradius of $\conv(Q\cup \{0\})$ in $L$ {(i.e., $s$ is the maximum radius of a ball in $L$ contained in $Q$)}.
\item[\rm (c)]  $V^Q(f)\le {R\, V_L(f) \le}R \, V(f)$, where $R$ denotes the 
          circumradius of $\conv(Q\cup \{0\})$, that is the radius of
          the unique smallest ball containing this set.  
\end{enumerate}
\end{lemma}

\begin{proof} The claim in (a) follows from the
definitions of $V_u(f)$ and $V^Q(f)$, in combination with \eqref{centerCond}. 
To verify (b), 
let $B_L(y,s)$ be a ball in $L$ included in $K:=\conv(\{0\}\cup Q)$. From the basic properties of support functions we get for $u\in S^{n-1}$ 
\begin{eqnarray*}
(h(Q,u))^+&=&h(K,u)=h(K,p_Lu)\\
&\geq& h(B(y,s),p_Lu)=y\cdot p_Lu+ s |p_L(u)|=y\cdot u+ s |p_L(u)|.
\end{eqnarray*}
Setting $u=\Delta_f(x)$ and integrating w.r.t.\ $|Df|$, 
equations  \eqref{centerCond} and \eqref{E-VL} imply  
\[
V^Q(f)\geq s V_L(f),
\]
as required. The proof of assertion (c) is analogous.
\end{proof}

For a non-empty compact set $Q\subset \R^n$ we define the \emph{$Q$-variation measure} $|\mu|_Q$ of the $\R^n$-valued Radon measure $\mu$ on the open set $\Omega\subset \R^k$ by 
\[
|\mu|_Q(A)=\sup\left\{\sum_{h=0}^\infty h(Q,\mu(E_h))^+: (E_1,E_2,\ldots) \text{
	forms a partition of } A\right\}
\]
for any Borel set $A\subset \Omega$. Using the subadditivity of the support function, it is easy to show that $|\mu|_Q$ is a positive Radon measure; one can for instance adapt the proof of \cite[Theorem 1.6]{AFP00} and observe that $Q\subset B(0,r)$ implies $|\mu|_Q\le |\mu|_{B(0,r)}=r|\mu|$ to prove finiteness on compact sets. 
The identity \eqref{eqB_L} shows that the following Proposition contains  Proposition \ref{propDirVar} as special case. 

\begin{proposition}[Basic properties of the variation with respect to $Q$]\mbox{}\\ \label{prop:VQ}
Let $Q\subset \R^n$ be non-empty and compact. 
\begin{enumerate}
\item[\rm (a)] 
 For $f\in \BV(\Omega)$ we have $V^{Q}(f,\Omega)=|Df|_Q(\Omega)$.  If, in addition, $f\in C^1(\Omega)$, 
 then 
\begin{align}\label{claimA} 
V^Q(f, \Omega)=\int_{\Omega} h\left(Q,\nabla f(x)\right)^+ dx.
\end{align}
\item[\rm (b)]  Assume that $\Omega=\R^n$ or that the origin is a relative interior point of $\conv Q$.  If $f_j\to f$ in $L^1(\Omega)$ then $V^Q(f,\Omega)\leq\liminf_{j\to\infty} V^Q(f_j,\Omega)$.
\item[\rm (c)]  If $(f_j)$ is a sequence converging strictly to $f\in\BV(\Omega)$, then  $V^Q(f_j,\Omega)\to V^Q(f,\Omega)$ as $j\to \infty$.
\end{enumerate}
\end{proposition}
\begin{proof}
In order to prove that $V^{Q}(f,\Omega)=|Df|_Q(\Omega)$ in (a), it is enough to show that if an $\R^n$-valued finite measure $\mu$ 
has density $g$ with respect to a positive measure $\nu$, then $|\mu|_Q$ has density $h(Q,g(\cdot))^+$ with respect to $\nu$, and apply this to $\mu=Df$, $\nu=|Df|$. With this notation, and exploiting that we may assume $0\in Q$, we have to prove 
\begin{equation}\label{eqDensity}
|\mu|_Q(B)=\int_B h(Q,g(x))\, \nu(dx)
\end{equation}
for all measurable sets $B$. The inequality $|\mu|_Q(B)\le\int_B h(Q,g) d\nu$ follows from the convexity, positive homogeneity and continuity of $h(Q,\cdot)$. To show the reverse inequality let $\varepsilon>0$ and choose a dense sequence $(z_h)$ in
$\conv Q$. Define 
\[
\sigma(x)=\min\{h\in {\mathbb N}: z_h\cdot g(x)\ge (1-\varepsilon)h(Q,g(x))\}, 
\]
and  the level sets $B_h=\sigma^{-1}(h)\cap B$, that form a partition of $B$. Then 
\begin{align*}
(1-\varepsilon)\int_B h(Q,g) d\nu&=\sum_h\int_{B_h} (1-\varepsilon)h(Q,g) d\nu
\le \sum_h\int_{B_h} z_h\cdot g(x)  d\nu\\&=\sum_h z_h\cdot \mu(B_h)\le \sum_h h(Q,\mu(B_h))
\le |\mu|_Q(B), 
\end{align*}
yielding \eqref{eqDensity}. If $f$ is also in $C^1(\Omega)$,  
$Df$ has Lebesgue-density $\nabla f$ and \eqref{eqDensity} with $\mu=Df$, $g=\nabla f$ and Lebesgue measure $\nu$ yields the second claim in (a). 

Let us show (b). 
We may assume that $\liminf_jV^Q(f_j,{\Omega})<\infty$, and
pass to a subsequence (again denoted by $(f_j)$) for which
$\lim_{j\to\infty}V^Q(f_j,{\Omega})<\infty$ exists. Set $L:=\Span Q$.  Except the trivial case $Q=\{0\}$, we always have $\dim L\geq 1$. 
If $\Omega=\R^n$ let $s>0$ be the inradius of $\conv(\{0\}\cup
Q)$ in $L$. Then, 
\begin{align}\label{eqbound}
V_L(f_j,\Omega)\le \frac 1s V^Q(f_j,\Omega)
\end{align}
 due to  Lemma~\ref{lem:VqV}.(b). If the origin is a relative interior point of $\conv Q$, there is $s>0$ such that 
  $sB_L\subset \conv Q$ and hence $V^Q(f_j,\Omega)\ge V^{sB_L}(f_j,\Omega)=sV^{B_L}(f_j,\Omega)=sV_L(f_j,\Omega)$, implying again \eqref{eqbound}. In either case, the sequence $V_L(f_j,\Omega)$ is bounded. 
Hence, by Proposition \ref{propDirVar}.(a),  $\mu_j:=p_L(Df_j)$, $j=1,2,3,\ldots$, are $L$-valued finite vector measures. We can show exactly as in the proof of \cite[Proposition~3.13]{AFP00} that $\mu_j\to\mu=p_L(Df)$ weakly* (we use the relative weak* compactness of $(\mu_j)$ and verify that any cumulative point of $(\mu_j)$ must agree with $p_L(Df)$). 
Note that the measures $\mu_j$, and $\mu$ 
have polar decompositions \eqref{polar}
$$d\mu_j=\frac{p_L\Delta_{f_j}}{|p_L\Delta_{f_j}|}\,d|\mu_j|,
\quad \mu=\frac{p_L\Delta_{f}}{|p_L\Delta_{f}|}\, d|\mu|,$$
and we can write
$$V^Q(f,\Omega)=\int_{\Omega} h\left(\conv(\{0\}\cup Q),\frac{p_L\Delta_{f}}{|p_L\Delta_{f}|}\right)\, d|\mu|,$$
and analogously with $f_j$ and $\mu_j$. Since the support function $h(\conv(\{0\}\cup Q),\cdot)$ is continuous and positively $1$-homogeneous, we may apply the Reshetnyak lower semi-continuity theorem \cite[Theorem~2.38]{AFP00} and we obtain $V^Q(f,{\Omega})\leq\liminf_jV^Q(f_j,{\Omega})$, as requested.

Assertion (c) follows directly from Lemma~\ref{L-Resh} with 
$h=h(\conv(\{0\}\cup Q),\cdot)$.
\end{proof}

\begin{remark} \rm If $\conv Q$ is symmetric w.r.t.\ the origin then assertion (b) of Proposition~\ref{prop:VQ} follows from \cite[Theorem~5.1]{AB:94}.
\end{remark}

\section{Dilation volumes} \label{s:dv}
Let 
$A\oplus  Q=\{a+q:a\in A,q\in Q\}$ be  the \emph{Minkowski sum} of the
sets $A$ and $Q$ in $\R^n$. For measurable $A$ and compact $Q\ne
\emptyset$ 
we are
interested in the volume 
\[
  \lambda_n((A\oplus Q)\setminus A)=\int_{\R^n} \left( \max_{u \in Q}
  \1_{A+u}(x)-\1_A(x)\right)^+ dx, 
\]
and therefore define more generally the functional 
\begin{align}\label{eqDefG}
  G(Q,f)= \int_{\R^n} \left(\sup_{u \in Q} f(x-u)-f(x)\right)^+ dx
\end{align}
for any measurable function $f$ on $\R^n$. Note that the family
$\{f(\cdot-u):u\in Q\}$ is a permissible class, and thus, 
$\sup_{u \in Q} f(\cdot-u)$ is Lebesgue-measurable; see e.g.~\cite[Appendix
C]{poll84} for a short summary or \cite[Section III]{DM:78} for
details.   
By definition, 
\begin{align}\label{eq:Ggeometric}
G(Q,\1_A)= 
  \lambda_n((A\oplus Q)\setminus A). 
\end{align}

Note that the mapping $f\mapsto G(Q,f)$ may depend  in 
general on the particular representation $f$ and, hence, 
cannot be considered as a mapping on $L^1$. When $Q$ is at most countable, 
independence of the representative is straightforward.

\begin{lemma}[Properties of $G(Q,\cdot)$ for countable $Q$] \label{G_smooth}
If {the compact set} $Q\ne \emptyset$ is at most countable then the mapping $G(Q,\cdot)$ 
is well-defined and lower semi-continuous on $L^1$.
Moreover, if $f_j=f* \rho_j$ is a mollification of $f\in L^1$ with
non-negative  $\rho$, then 
\begin{align}\label{GG}
G(Q,f_j)\le G(Q,f)
\end{align}
and thus 
$G(Q,f_j)\to G(Q,f)$, as $j\to \infty$. 
\end{lemma}

\begin{proof}
For integrable $f$, let $f^Q(x)=\sup_{u\in Q} f(x-u)$. If $g$ is
another representative of the $L^1$-equivalence class of $f$, then
$f=g$ outside a set $N$ of Lebesgue measure zero. 
Then, $f^Q=g^Q$ outside the set $N\oplus Q$, the latter being a
Lebesgue-null set as $Q$ is at most countable. Hence $G(Q,\cdot)$ is
well-defined on $L^1$.

 To show the semi-continuity, let $(f_j)$ be a sequence that converges to $f$ in $L^1$. This implies that $(f_j)$ converges in measure and if we consider a subsequence of $(f_j)$ such that the limit inferior (of $(G(Q,f_j))$) becomes an ordinary limit, there is a sub-subsequence $(f_{j'})$ that converges outside a Lebesgue-null set $N$. 
 	As $Q$ is at most countable, $M=N\oplus (Q\cup\{0\})$ is a Lebesgue-null set, and  we have that $\lim_{j\to \infty} f_{j'}(x-u)=f(x-u)$ for all $u\in Q\cup\{0\}$ and  $x\not\in M$. Fatou's lemma and the lower semi-continuity of the supremum operation now yield
 	\begin{align*}
 	\liminf_{j\to\infty} G(Q,f_j)&\ge \int_{\R^n} \liminf_{j\to\infty}\sup_{u \in Q\cup\{0\}}\left( f_{j'}(x-u)-f_{j'}(x)\right) dx\\
 	&\ge \int_{\R^n\setminus M} \sup_{u \in Q\cup\{0\}}\liminf_{j\to\infty}\left( f_{j'}(x-u)-f_{j'}(x)\right) dx\\
 	&= \int_{\R^n\setminus M} \sup_{u \in Q\cup\{0\}}\left( f(x-u)-f(x)\right) dx
 	         \\
 	         &=G(Q,f).
 	\end{align*}
It remains to prove \eqref{GG}. We may  assume without loss of 
generality that $0\in Q$, as $G(Q,f)=G(Q\cup \{0\},f)$. Then, the positive part can be dropped in
 the definition of $G(Q,f)$.  We have
\begin{eqnarray*}
G(Q,f_j)&=&\|\sup_{u\in Q}(f_j(\cdot -u)-f_j)\|_1\\
&=&\|\sup_{u\in Q}\left[(f(\cdot -u)-f)*\rho_j\right]\|_1\\
&\leq&\left\|\left[\sup_{u\in Q}(f(\cdot -u)-f)\right]*\rho_j\right\|_1\\
&=&\|\sup_{u\in Q}(f(\cdot -u)-f)\|_1\\
&=&G(Q,f).
\end{eqnarray*}
We have used the inequality
\[
\sup_{u\in Q}[g_u*h]\leq [\sup_{u\in Q}g_u]*h
\]
valid for any integrable functions $h\ge 0$ and $g_u, u\in Q$.
\end{proof}

\begin{proposition}   \label{PG}
{  If $f\in C^1\cap BV$,
$Q\subset \R^n$ is non-empty and compact, 
and $r> 0$ then
\begin{equation} \label{GV}  
  \liminf_{r\to 0+} \frac1r G(rQ,f)\ge V^{-Q}(f).
\end{equation}
If $Q$ is in addition at most countable, \eqref{GV}  holds for any $f\in \BV$. 
}
\end{proposition}
\begin{proof}
Assume first that $f\in C^1\cap \BV$.
Using the function
\[
  g_x(r)=\max_{u\in Q} f(x-ru)-f(x),\quad r\geq 0,
\]
we may write 
\begin{align}\label{latenight}
  \frac1r G(rQ,f)= \int_{\R^n} \left( \frac1r g_x(r)\right)^+ dx.
\end{align}
Fix $x\in\R^n$. As $f$ is Lipschitz in a neighborhood of $x$ with Lipschitz
constant $M_x$, say, $g_x$ is Lipschitz in a neighborhood $V$ of $0$ with constant bounded by 
$M_x\max_{u\in Q}|u|$. Hence $g_x$ is differentiable almost
everywhere in $V$,   
this derivative is essentially bounded uniformly in $V$, and 
\[
g_x(r)=g_x(0)+\int_0^r g'_x(s)ds=r\int_0^1 g'_x(rs)ds.
\]
Inserting this into \eqref{latenight}, and using 
the fact that  $g'_x(rs)$ coincides almost everywhere with the right sided derivative 
$g'_x(rs+)$, this gives
\begin{align}\label{eq:GrQ}
\frac1r G(rQ,f)= \int_{\R^n} \left( \int_0^1 g'_x(rs+)ds\right)^+ dx.
\end{align}

To determine the limit inferior we first fix $x\in\R^n$. For every $r>0$ there  
is some $v_r\in Q$ with 
\[
g_x(r)=f(x-rv_r)-f(x).
\]
Thus, for all $u\in Q$, $f(x-ru){-f(x)}\le f(x-rv_r){-f(x)}$  and division with $r$ and taking the limit $r\to0_+$  yields
\[
  (-\nabla f(x))\cdot u \le  (-\nabla f(x))\cdot v,
\]
for all $u\in Q$, where $v\in Q$ is any accumulation point of a subsequence of $(v_r)$. Hence,  
\begin{align}\label{eq:vGood}
h(Q,-\nabla f(x))=  (-\nabla f(x))\cdot v.
\end{align}
A lower bound for $g_x'(r+)$ is now obtained from

\begin{align*}
g_x'(r+)&=\lim_{s\to 0+}\frac 1s (g_x(r+s)-g_x(r))\\
&\geq\frac 1s\lim_{s\to 0+}\left((f(x-(r+s)v_r)-f(x))-(f(x-rv_r)-f(x))\right)\\
&\ge  (-\nabla f(x-rv_r))\cdot v_r.
\end{align*}
Considering a subsequence of $(v_r)$ such that the limit inferior becomes a limit and is converging to some $v\in Q$, we can take the limit and get from \eqref{eq:vGood} that 
\[
\liminf_{r\to 0+}g_x'(r+)
\ge ( -\nabla f(x))\cdot v=h(Q,-\nabla f(x)). 
\]
As $g_x'(r+)$ is essentially bounded by $M_x\max_{u\in Q}|u|$ in $V$, dominated convergence implies
$$\liminf_{r\to 0+}\int_0^1g_x'(rs+)\, ds\geq h(Q,-\nabla f(x)).$$
 This
can be used in \eqref{eq:GrQ}, after applying Fatou's lemma, to obtain
\[
\liminf_{r\to 0+} \frac1r G(rQ,f)\ge V^{-Q}(f). 
\]
This yields the assertion for continuously differentiable $f$. 
\medskip 

Let now $f$ be a general function of bounded variation, and 
let $f_j=f*\rho_j$ be smooth mollifications 
of $f$ with mollifiers $\rho_j\ge 0$ (cf.\ Section~\ref{sec:prelim}). 
Let $Q$ be 
non-empty and  at most countable. Then inequality \eqref{GV} 
	holds for all $f_j$ and by Lemma~\ref{G_smooth} and Proposition~\ref{prop:VQ}.(b) also for $f$. This completes the proof.
\end{proof}

The arguments in the proof of \cite[Proposition~11]{Gal:10} show that 
	\[
	\lim_{r\to 0_+} \frac 1r G(rQ,f)=V^{-Q}(f)
	\]
	when $Q=\{u\}, u\not=0$, which is a version of \eqref{covariogram_e} for $\BV$ functions $f$.  One might thus expect that the limit inferior in  \eqref{GV}
is indeed an ordinary limit, and equality holds, for at most countable sets $Q$. However, when  $Q$ is 
infinite, this need not be true. In the following we give a counterexample where  $f$ is the indicator function of a compact set of finite perimeter. This example 
is adapted from the known example of a set of positive reach with infinite outer Minkowski content, see e.g.\  
 \cite[pp. 109f]{AFP00}.

\begin{example}  \label{ex:1}  \rm
Let $n\ge2$. 
For every $m\in\bbN$ define the open annulus
\[
R_m=\INt\left(B\left(0,\frac{1}{m}\right)\setminus B\left(0,\frac{1}{m+1}\right) \right), 
\]
and choose a finite set $A_m\subset R_m$ with 
\[
  R_m\subset A_m\oplus B(0,(2^mm)^{-1}).
\]
Let $(r_m)$ be a sequence of positive numbers and set $A=\{0\}\cup \bigcup_{m=1}^\infty \left( A_m\oplus B(0,r_m)\right)$. 
The sequence $(r_m)$ can be chosen in such a way that 
\[
\cH^{n-1}(\partial A)\le \sum_{m=1}^\infty (\#A_m)\cH^{n-1}(\partial B(0,1)) r^{n-1}_m <\infty,
\]
(here we use the assumption $n\ge 2$),  $A_m\oplus B(0,r_m)\subset R_m$, 
and 
\begin{align}\label{eq:smallLambda}
\lambda_n(A_m)\le\sum_{m=1}^\infty (\# A_m) \lambda(B(0,1)) r_m^n<\frac{ \lambda_n(R_m)}2
\end{align}
for all $m\in \bbN$. In particular, $A$ is a compact set of  finite perimeter. In a similar way, choose finite sets 
$Q_m\subset B(0,1/m)$ with $Q_m\oplus B(0,r_m)\supset B(0,1/m)$ for all $m\in\bbN$, and set 
$Q=\{0\}\cup \bigcup_{m=1}^\infty Q_m$. Then $Q$ is a compact countable subset of the unit ball. 
For $0<r<1/2$ let $m$ be such that $2^{-m}<r\le 2^{-m+1}$. Then 
\begin{align*}
(  A \oplus rQ)\setminus A&\supset [(A \oplus rQ)\setminus A]\cap R_m\\
&\supset  \big[\big( A_m\oplus B(0,r_m) \oplus rQ\big)\cap R_m\big]\setminus A_m\\
&\supset \big[\big(A_m\oplus r(Q\oplus B(0,r_m))\big)\cap R_m\big]\setminus A_m\\
&\supset \big[ \big(A_m\oplus B(0,(2^mm)^{-1})\big)\cap R_m\big]\setminus A_m\\
&=R_m\setminus A_m.
\end{align*}
It follows from \eqref{eq:Ggeometric} and \eqref{eq:smallLambda} that there is a constant $c>0$ with 
\[
\frac1{r}G(rQ,\1_A)\ge \frac{\lambda_n(R_m)}{2r}\ge c\frac{(\log_2\frac1r)^{-(n+1)}}{r}\to\infty,
\]
as $r\to 0_+$. In particular, $\frac1{r}G(rQ,\1_A)$ does not converge to $V^{-Q}(\1_A)\le \cH^{n-1}(\partial A)< \infty$ (we use here Lemma~\ref{lem:VqV}(c)).
\end{example}

We will show now that the desired convergence result is true when $f$ is the indicator of a set of finite perimeter and 
$Q$ is finite. This requires some auxiliary lemmas. We recall the notation $\cF_{u+}A$, $\cF_{u-}A$ introduced in Section~\ref{sec:prelim}.

\begin{lemma}   \label{l-aux}
Let $0\neq u\in \R^n$ and $r>0$ be given.
\begin{enumerate}
\item[{\rm (i)}] If $f\in\BV$ and $U\subset\R^n$ is open then
$$\int_U(f(x)-f(x+ru))^+\, dx\leq r V^{\{-u\}}(f,U\oplus(0,ru)).$$
\item[{\rm (ii)}] If $A\subset\R^n$ has finite perimeter then
$$\lambda_n\left(\left\{x\in A:\, x+ru\not\in A,\, [x,x+ru]\cap\cF_{u+} A=\emptyset\right\}\right)=0.$$
\item[{\rm (iii)}]
If $A$ is as in (ii) and $0<s<1$ then
$$\lambda_n\left(\left\{ x\in A:\, x+sru\not\in A,\, x+ru\in A\right\}\right) =o(r),\quad r\to 0.$$
\end{enumerate}
\end{lemma}

\begin{proof}  
In fact, (i) is a local and signed version of \cite[Proposition~11]{Gal:10} and we proceed with a similar proof. 
If $f$ belongs to $C^1(U\oplus(0,ru))\cap \BV(U\oplus(0,ru))$ then
$$f(x)-f(x+ru)=\int_0^1r\left(-\frac{\partial f}{\partial u}(x+tru)\right)\, dt\leq\int_0^1r\left(\frac{\partial f}{\partial u}(x+tru)\right)^-\, dt$$
for all $x\in U$, and, applying Fubini's theorem and \eqref{claimA}, we get
$$\int_U(f(x)-f(x-ru))^+\, dx\leq \int_0^1 r \int_U \left(\frac{\partial f}{\partial u}(x+tru)\right)^-\, dx\, dt \leq r V^{\{-u\}}(f,U\oplus(0,ru)).$$
The case $f\in\BV$ can be shown by strict approximation: By Proposition 
\ref{prop:Var}.(c) there is a sequence $f_j$ in $C^\infty(U\oplus(0,ru))\cap \BV(U\oplus(0,ru))$ converging strictly to $f$ on $U\oplus(0,ru)$. Now, (i) holds with $f$ replaced by $f_j$, and taking the limit $j\to\infty$ it also holds for $f$ due to Proposition \ref{prop:VQ}.(c) and since $f_j\stackrel{L_1}{\to}f$ implies $(f_j(\cdot)-f_j(\cdot+ru))^+\stackrel{L^1}{\to}(f(\cdot)-f(\cdot+ru))^+$.

We will show (ii) by contradiction, i.e., assume that $\lambda_n(Z)>0$, where
$$Z:=(A\setminus(A-ru))\setminus (\cF_{u+} A\oplus [0,-ru]).$$
Note that, in particular, $(D_u\1_A)^-(Z\oplus [0,ru])=0$ (cf.\ \eqref{Du-}). Since the measure $(D_u\1_A)^-$ is outer regular, we can find an open set $V\supset Z\oplus[0,ru]$ such that
$(D_u\1_A)^-(V)<r^{-1}\lambda_n(Z)$. Let, further, $U\supset Z$ be an open set such that
$U\oplus[0,ru]\subset V$ (we can set $U=V\ominus[0,ru]$, where $\ominus$ is the Minkowski subtraction, and use \cite[(3.15)]{sch93}). Then, applying (i) with $f=\1_A$, we obtain
\begin{eqnarray*}
\lambda_n(Z)\leq\lambda_n((A\setminus(A-ru))\cap U)
&=&\int_U(\1_A(x)-\1_A(x+ru))^+\, dx\\
&\leq&r V^{\{-u\}}(\1_A,U\oplus[0,ru])\\
&=& r (D_u\1_A)^-(U\oplus [0,-ru])\\ 
&\leq& r (D_u\1_A)^-(V) <\lambda_n(Z),
\end{eqnarray*}
a contradiction completing the proof of (ii).

In order to prove (iii), we apply (ii) and get
$$\lambda_n\left(\left\{ x\in A:\, x+sru\not\in A,\, x+ru\in A\right\}\right)
\leq \lambda_n(\{z:\, \# ([x,x+ru]\cap \cF A)\geq 2\}).$$
The last measure is of order $o(r)$ since $\cF A$ is $\cH^{n-1}$-rectifiable (see, e.g., \cite[Lemma~1]{rata04}), and the proof is finished.
\end{proof}

Let now a set $A\subset\R^n$ of finite perimeter and a finite set $Q=\{u_0=0,u_1,\ldots,u_k\}\subset\R^n$ be given.
To any $x\in\cF A$ we assign the (unique) smallest number $0\leq i(x)\leq k$ for which $\nu_A(x)\cdot u_{i(x)}=\max_j\nu_A(x)\cdot u_j$, and we consider the partition
$$\cF A=\bigcup_i \partial_iA$$
with $\partial_iA:=\{x\in\cF A:\, i(x)=i\}$, $i=0,\ldots,k$.
Note that $\partial_iA\subset\cF_{u_i+}A$, $i=1,\dots,k$, by definition.
Denoting further
$$A_{Q,r}:=\bigcup_{i=0}^k(\partial_iA\oplus [0,ru_i]),$$
we have, using Fubini's theorem and the area formula for the orthogonal projection of $A_i$ onto $u_i^\perp$ (see \cite[Theorem~2.91]{AFP00}),
\begin{eqnarray}  \label{E_decomp}
\lambda_n(A_{Q,r})&\leq&\sum_{i=0}^kr\cH^{n-1}(p_{u_i^\perp}(\partial _iA))|u_i|
=\sum_{i=0}^k r\int_{\partial_iA}|u_i\cdot\nu_A(x)|\, \cH^{n-1}(dx)\nonumber\\
&=&r\int_{\cF A}\max_i(u_i\cdot\nu_A(x))\,\cH^{n-1}(dx)\nonumber\\
&=&rV^{-Q}(\1_A).
\end{eqnarray}

\begin{lemma}  \label{L_inequal}
Let $A\subset\R^n$ have finite perimeter and let $0\in Q\subset\R^n$ be finite. Then we have
$$\lambda_n\Big( ((A\oplus rQ)\setminus A)\setminus A_{Q,r}\Big)=o(r),\quad r\to 0.$$
\end{lemma}

\begin{proof}
First, we shall show that it is sufficient to consider sets $Q=\{0,u_1,\ldots,u_k\}$ such that for all $1\leq i<j\leq k$, the vectors $u_i,u_j$ are either linearly independent, or linearly dependent, but pointing in opposite directions. To see this, consider a larger set $Q'=Q\cup\{su_k\}$ with some $0<s<1$. We have clearly $A_{Q',r}=A_{Q,r}$, $r>0$, and
\begin{eqnarray*}
\lambda_n((A\oplus rQ')\setminus(A\oplus rQ))
&\leq&\lambda_n ((A+sru_k)\setminus(A\oplus\{0,ru_k\}))\\
&\leq&\lambda_n(\{z:\, z\not\in A,\, z+sru_k\in A,\, z+ru_k\not\in A\}),
\end{eqnarray*}
and the last expression is of order $o(r)$ by Lemma~\ref{l-aux}.(iii) applied to the complement of $A$.

Any point $z\in ((A\oplus rQ)\setminus A)\setminus A_{Q,r}$ has the following properties: $z\not\in A$, $z-ru_i\in A$ for some $1\leq i\leq k$ and $[z-ru_j,z]\cap\partial_jA=\emptyset$ for all $1\leq j\leq k$. By Lemma~\ref{l-aux}.(ii), $\lambda_n$-almost all such points $z$ have the additional property that there exists a point $x\in [z-ru_i,z]\cap\cF_{u_i+}A$ and, clearly, this $x$ must belong to $\partial_jA$ for some $j\neq i$, $j\geq 1$. Hence,
$$\lambda_n\Big( ((A\oplus rQ)\setminus A)\setminus A_{Q,r}\Big)\leq\sum_{j\neq i}\lambda_n(V_{ij}^r)$$
with
$$V_{ij}^r:=\{z:\, [z-ru_i,z]\cap F_{ij}\neq\emptyset,\,[z-ru_j,z]\cap F_{ij}=\emptyset \},$$
where
$$F_{ij}:=\partial_jA\cap\cF_{u_i+}A.$$
It is thus enough to show that $\lambda_n(V_{ij}^r)=o(r)$ for any $1\leq i\neq j\leq k$.

Note that if $u_j=-su_i$ for some $s>0$ then $F_{ij}=\emptyset$ (indeed, in this case $u_i\cdot\nu_A(x)>0$ implies $u_j\cdot\nu_A(x)<0<u_i\cdot\nu_A(x)$). Thus, we can assume in the sequel that $u_i,u_j$ are linearly independent.

Applying the Fubini's theorem and the generalized area formula \cite[\S3.2.22]{Fed69} with the orthogonal projection $p_{u_i^\perp}|F_{ij}$ (note that $F_{ij}\subset\cF A$ is countably $(n-1)$-rectifiable and its Jacobian $J_{n-1}(p_{u_i^\perp}|F_{ij})$ is at most $1$), we get
\begin{eqnarray*}
\lambda_n(V_{ij}^r)&=&\lambda_n(V_{ij}^r\cap(F_{ij}\oplus[0,ru_i]))\\
&=&\int_{u_i^\perp}\lambda_1\left(V_{ij}^r\cap(F_{ij}\oplus[0,ru_i])\cap (y+\Span (u_i))\right)\, \lambda_{n-1}(dy)\\
&\leq&\int_{u_i^\perp}\sum_{x\in F_{ij}\cap(y+\Span (u_i))}\lambda_1(V_{ij}^r\cap[x,ru_i])\, \lambda_{n-1}(dy)\\
&=&\int_{F_{ij}}J_{n-1}(p_{u_i^\perp}|F_{ij})(x)\, \lambda_1(V_{ij}^r\cap[x,x+ru_i])\, \cH^{n-1}(dx)\\
&\leq&\int_{F_{ij}}\lambda_1(V_{ij}^r\cap[x,x+ru_i])\, \cH^{n-1}(dx).
\end{eqnarray*}
Hence we have
$$r^{-1}\lambda_n(V_{ij}^r)\leq\int_{F_{ij}}\varphi^r(x)\, \cH^{n-1}(dx),$$
where
$$\varphi^r(x):=r^{-1}\lambda_1(V_{ij}^r\cap[x,x+ru_i]),\quad x\in F_{ij}.$$
We will show that
\begin{equation}  \label{Eq_lim}
\lim_{r\to 0}\varphi^r=0\quad\cH^{n-1}-\text{a.e. on }F_{ij}.
\end{equation}
Applying then the Lebesgue dominated convergence theorem (note that $|\varphi^r(x)|\leq|u_i|$ for any $x$) we obtain $\lambda_n(V_{ij}^r)=o(r)$, proving the lemma.

We will verify \eqref{Eq_lim}. Since $F_{ij}$ is countably $(n-1)$-rectifiable, the approximate tangent cone $\Tan^{n-1}(F_{ij},x)$ is a hyperplane at $\cH^{n-1}$-a.a. $x\in F_{ij}$ by \cite[\S3.2.19]{Fed69}, and we thus get $\Tan^{n-1}(F_{ij},x)=\nu_A(x)^\perp$ at $\cH^{n-1}$-a.a. $x\in F_{ij}$ by \cite[Theorem~3.59]{AFP00}. (Concerning rectifiability, we use the terminology from \cite{AFP00} which is slightly different from \cite{Fed69}.)

Denote $L:=\Span(u_i,u_j)$. We apply the generalized co-area formula \cite[\S3.2.22]{Fed69} to the mapping $f:=p_{L^\perp}|F_{ij}:F_{ij}\to L^\perp$. We get that $f^{-1}\{z\}=F_{ij}\cap(z+L)$ is countably $1$-rectifiable for $\cH^{n-2}$-a.a.\ $z\in L^\perp$ and, thus, for $\cH^1$-a.a.\ $x\in F_{ij}\cap(z+L)=F_{ij}\cap(x+L)$, the one-dimensional Lebesgue density $\Theta^1(F_{ij}\cap(x+L),x)=1$ (cf.\ \cite[\S3.2.19]{Fed69}) and
\begin{align} \label{E_Fij}
\Tan^1(F_{ij}\cap(x+L),x)=\nu_A(x)^\perp\cap L.
\end{align}
Let $N$ denote the set of all $x\in F_{ij}$ for which \eqref{E_Fij} is not true. We have $\cH^1(N\cap f^{-1}\{z\})=0$ for $\cH^{n-2}$-a.a.\ $z\in L^\perp$, hence, again by the co-area formula,
$$\int_N J_{n-2}f(x)\, \cH^{n-1}(dx)=\int_{L^\perp}\cH^1(N\cap f^{-1}\{z\})\, \cH^{n-2}(dz)=0.$$
As $J_{n-2}f(x)\neq 0$ for $x\in F_{ij}$ (recall that both $\nu_A(x)\cdot u_i$ and $\nu_A(x)\cdot u_j$ are positive if $x\in F_{ij}$), we have $\cH^{n-1}(N)=0$, hence, \eqref{E_Fij} is true for $\cH^{n-1}$-a.a.\ $x\in F_{ij}$.

Fix now a point $x\in F_{ij}$ for which \eqref{E_Fij} holds, set
\begin{align*}
q:=&\frac{\nu_A(x)\cdot u_i}{\nu_A(x)\cdot u_j}\in (0,1],\\
w:=&u_i-qu_j\in \nu_A(x)^\perp\cap L,
\end{align*}
and choose an $\ep>0$.
Note that small positive multiples of the vector $w$ lie in the open triangle
$$C:=\{tu_i-su_j:\, 0<\frac{s}{q+\ep}<t<1\}$$
and, consequently, also
$$\Theta^1(F_{ij}\cap(x+rC),x)=\tfrac 12$$
for any $r>0$. If $\pi$ denotes the projection from $x+L$ onto $x+\Span(u_i)$ along $u_j$, we get as a consequence that
$$\Theta^1(\pi(F_{ij}\cap(x+rC)),x)=\tfrac 12.$$
On the other hand, if $z=x+tu_i\in V_{ij}^r$ for some $0<t<\frac{r}{q+\ep}$ then $z\not\in\pi(F_{ij}\cap(x+rC))$ and, consequently,
\begin{eqnarray*}
\lambda_1(V_{ij}^r\cap[x,x+ru_i])&\leq&
\left(r-\frac{r}{q+\ep}\right)^++\lambda_1\left([x,x+ru_i]\setminus\pi(F_{ij}\cap(x+rC))\right)\\
&\leq& \ep r+o(r).
\end{eqnarray*}
Since $\ep>0$ can be arbitrarily small, we obtain \eqref{Eq_lim} and the proof is finished.
\end{proof}

\begin{corollary}\label{cor:limsup}
Let $A\subset\R^n$ have finite perimeter and let $\emptyset\ne Q\subset \R^n$ be finite. Then
$$\limsup_{r\to 0_+}\frac 1rG(rQ,\1_A)\leq V^{-Q}(\1_A).$$
\end{corollary}

\begin{proof}{
	As both sides of the stated equality remain unchanged when $Q$ is replaced by $Q\cup \{0\}$ we may assume that $0\in Q$. The claim then follows} from \eqref{E_decomp} and Lemma~\ref{L_inequal}.
\end{proof}

Proposition~\ref{PG} and Corollary~\ref{cor:limsup} yield already our main result:

\begin{proposition}\label{prop:finiteQ}
	Assume that  $A\subset\R^n$ has finite perimeter. If
 $\emptyset\ne Q\subset\R^n$ is finite then 
$$\lim_{r\to 0+} r^{-1} G(rQ,\1_A)=V^{-Q}(\1_A).$$
\end{proposition}

\begin{proof}[Proof of Theorem~\ref{thm:main}]
The first statement, \eqref{eq:main1}, follows directly from Proposition~\ref{prop:finiteQ} in combination with \eqref{eq:Ggeometric} and \eqref{mixedvol}.  If  $0\in Q$ and $\lambda_n(A)<\infty$ then \eqref{eq:main2} holds as it then coincides
with  \eqref{eq:main1}. It is thus enough to
show that the two sides of \eqref{eq:main2} do not change, 
when $Q$ is replaced by a translation $Q-x$ with $x\in Q$. 
This is trivial for the left hand side and follows, using
\eqref{centerCond}, also for the right hand side.  
Hence \eqref{eq:main2} also 
holds without the additional restriction  $0\in Q$. 
\end{proof}

\section{An application: Contact distributions of stationary random
  sets}
  \label{s:appl}

In this section we apply the geometric results to random sets; 
see the book \cite{SW:08} for details on random closed sets, and \cite{Gal:LachRey:15} for random measurable sets in
$\R^n$. 
Galerne and Lachi\`eze-Rey \cite{Gal:10,Gal:LachRey:15}  define the mean covariogram of a random measurable set
and discusses its properties. With the results of the previous section,
 similar relations for the \emph{mean generalized dilation volume} with a finite structuring element could be
 established. We will not do so here, but instead present an 
approximation  of the contact distribution function of a random  set at
zero, as the contact distribution function is an important summary statistics in applications. 
\medskip

We recall the notion of a {\it random measurable set} in $\R^n$. Let $\cM$ denote the space of all Lebesgue measurable subsets of $\R^n$ modulo set differences of Lebesgue measure zero, equiped with the topology of $L^1_{\loc}$ convergence of the indicator functions. If $\cB(\cM)$ denotes the corresponding Borel $\sigma$-algebra, $(\cM,\cB(\cM))$ is a standard Borel space, and a random measurable set (RAMS) is a measurable mapping
$$Z:(\Omega,\Sigma,\Pr)\to (\cM,\cB(\cM))$$
from a probability space $\Omega$.
(As remarked in \cite[Remark~1]{Gal:LachRey:15}, the random sets of finite perimeter from \cite{Rataj:15} are just random measurable sets with finite specific perimeter.)
We restrict attention to {\it stationary} random measurable sets $Z$ in $\R^n$
(that is, random measurable sets with translation-invariant  distribution). 

If $Z$ is a stationary random {\it closed} set with volume fraction $\overline{p}=\Pr[0\in Z]<1$, its \emph{contact distribution function} (sometimes called hit distribution function) with a compact structuring element $Q\subset\R^n$ is defined by 
\begin{align}  \label{def_H_Q}
  H_Q(r)=\Pr(Z\cap rQ\ne \emptyset\mid 0\not\in Z),\quad r\geq 0. 
\end{align}
If $\overline p=1$, we set  $H_Q(r)=1$. For convex $Q$ with $0\in
Q$ and $\overline p<1$, $H_Q(\cdot)$ coincides with the  function
\[
  \tilde H_Q(r)=\Pr(d_Q(Z)\le r\mid 0\not\in Z),
\]
where $d_Q(Z)=\min\{t\ge 0: Z\cap tQ\ne \emptyset\}$. In general  we have 
\[
 \tilde H_Q(r)=H_{\Star Q}(r),
\]
where $\Star Q=\bigcup_{y\in Q}[0,y]$ is the star-hull of $Q$
with respect to $0$. 

Notice that \eqref{def_H_Q} does not give sense if $Z$ is a stationary RAMS since $[0\in Z]$ or $[Z\cap rQ\neq\emptyset]$ are not events (measurable subsets of $\Omega$) any more. (Indeed, one cannot determine whether $0$ belogs to $Z(\omega)$ since $Z(\omega)$ is given only up to measure zero.) Nevertheless, under stationarity, and for finite $Q$, we can give a meaning to \eqref{def_H_Q} as follows. We consider the \emph{shift randomization} $\widetilde{Z}$ of $Z$ defined on the larger probability space $\widetilde{\Omega}:=\Omega\times[0,1]^n$ with $\widetilde{\Pr}:=\Pr\otimes(\lambda_n|_{[0,1]^n})$ and $\widetilde{\Sigma}$ being the completion of the product $\sigma$-algebra $\Sigma\otimes \cB(\cR^n)$ as follows:
$$\widetilde{Z}(\omega,x):=Z(\omega)-x,\quad (\omega,x)\in\widetilde{\Omega}.$$ 
By stationarity, we get the equality in distribution, $\widetilde{Z}\stackrel{d}{=}Z$. In Lemma~\ref{L_rams} below, we show that $[0\in\widetilde{Z}]$ and $[\widetilde{Z}\cap rQ\neq\emptyset]$ are  random events, and we can define the volume fraction of $Z$ as $\overline{p}:=\widetilde{\Pr}[0\in\widetilde{Z}]$ and the contact distribution function $H_Q(r)$ of $Z$ using \eqref{def_H_Q}, where $\widetilde{\Pr},\widetilde{Z}$ are used instead of $\Pr,Z$. This 
contact distribution function satisfies 
\[
H_Q(r)=1-\frac{1-\bbE\lambda_n\left( (Z\oplus (-rQ\cup \{0\}))\cap [0,1]^n\right)}{1-\bbE\lambda_n(Z\cap [0,1]^n)},
\]
$r\ge 0$, 
which is a known representation of $H_Q$ when $Z$ is a RACS; cf.~\cite[p.~44]{SW:08}.

\begin{lemma}  \label{L_rams}
Let $Z$ be a stationary RAMS in $\R^n$ and $\widetilde{Z}$ its shift randomization. Then $[x\in\widetilde{Z}]$ is a random event (i.e., a measurable subset of $\widetilde{\Omega}$) for any $x\in \R^n$. If $Q\subset\R^n$ is at most countable then $[\widetilde{Z}\cap Q\neq\emptyset]$ is also a random event.
\end{lemma}

\begin{proof}
According to \cite[Proposition~1]{Gal:LachRey:15}, $Z$ admits a measurable graph representative, i.e., a subset $Y\subset\Omega\times\R^n$ measurable w.r.t.\ $\Sigma\otimes\cB(\R^n)$ such that for a.a.\ $\omega\in\Omega$, $\lambda_n(Z(\omega)\Delta Y_\omega)=0$, where $Y_\omega:=\{x\in\R^n:\, (\omega,x)\in Y\}$. Then we have by Fubini's theorem
\begin{align*}
\widetilde{\Pr}\left([0\in\widetilde{Z}]\Delta \left(Y\cap(\Omega\times[0,1]^n)\right)\right)
=\int_{\Omega}\lambda_n\left( (Z(\omega)\Delta Y_\omega)\cap [0,1]^n\right)\, \Pr(d\omega)=0.
\end{align*}
Since $Y$ is product-measurable and $\widetilde{\Sigma}$ is complete, also $[0\in\widetilde{Z}]$ is in $\widetilde{\Sigma}$. When $x\in \R^n$ is given, 
$Z-x$ is a RAMS, and thus $[x\in \tilde Z]=[0\in \tilde Z-x]=[0\in \widetilde{Z-x}]$ is measurable. The second assertion now follows from this and the fact that 
\[
[\tilde Z\cap Q=\emptyset]=\bigcap_{u\in Q}[u\not \in \tilde Z],
\]
and  the proof is finished.
\end{proof}

Let $Z$ be a stationary RAMS.
If $Z$ has a.s.~locally finite perimeter (i.e. $P(Z,\Omega)<\infty$ almost surely
for all bounded open sets $\Omega$), 
its derivative, the random
$\R^n$-valued Radon measure  $D\1_Z$ exists, and inherits stationarity
from $Z$. Hence, $|D\1_Z|$ is a stationary nonnegative Radon measure, and
there is  $\ova(Z)\in [0,\infty]$ such that
$\bbE|D\1_Z|=\ova(Z)\lambda_n$. The constant  $\ova(Z)$ is called the
\emph{specific perimeter} of $Z$ (see \cite{Gal:10,Rataj:15}) and we extend it by $\ova(Z):=\infty$
to those $Z$ which do not almost surely have locally bounded
variation. 
By definition, for any open  $\Omega\subset \R^n$ 
the random variable $P(Z,\Omega)$ is an unbiased
estimator of $\ova(Z)\lambda_n(\Omega)$. 
The specific perimeter  can also be obtained as usual  by an
averaging process over increasing windows. 

\begin{lemma}\label{lem:W}
Let $Z$ be a stationary RAMS. Then 
\begin{align}\label{eqW}
  \ova(Z)=\lim_{r\to \infty}\frac{\bbE P(Z\cap rW)}{\lambda_n(rW)},
\end{align}
where $W\subset \R^n$ is a compact convex set with positive volume. 
\end{lemma}
\begin{proof}
Due to stationarity, we may assume $0\in \text{int}W$. 
For $\Omega=r(\text{int}W)$, we have 
\[
  (\partial^*Z)\cap \Omega\subset  \partial^*(Z\cap rW)\subset 
[(\partial^*Z)\cap \Omega]\cup r \partial W. 
\]
Applying the $(n-1)$st Hausdorff-measure, and taking expectations, yields
\begin{align}\label{eq:lll} 
\bbE \cH^{n-1}(\partial^*Z\cap \Omega)\le
 \bbE P(Z\cap rW)\le
\bbE \cH^{n-1}(\partial^*Z\cap \Omega)
+r^{n-1}\cH^{n-1}(\partial W). 
\end{align}
If $Z$ has a.s.~locally finite perimeter, a 
comparison with the definition of $\ova(Z)$ yields  \eqref{eqW}.
Otherwise, there is  some open bounded set $\tilde\Omega$ such that 
$\cH^{n-1}(\partial^*Z\cap\tilde\Omega)=\infty$  
with positive probability. Then the
expectation on the left hand side of \eqref{eq:lll} is infinite for
all sufficiently large $r$, and the limit in \eqref{eqW} equals
infinity, as required. 
\end{proof}

If $Z$ is a stationary RAMS with $\ova(Z)<\infty$,  
then, for almost all realizations of $Z$, the 
generalized inner normal $\Delta_{\1_Z}(z)$ is defined for
$\cH^{n-1}$-almost all $z\in \partial^* Z$. 
Consider the random measure on $\R^n\times S^{n-1}$ given by 
\[
\Psi(B\times U)=\cH^{n-1}\{ z\in\partial^*Z\cap B:\, -\Delta_{\1_Z}(z)\in U\},\quad B\times U\in{\mathcal B}(\R^n\times S^{n-1});
\]
cf.~\cite[Proposition 4.2]{Rataj:15}. Since $\Psi$ is stationary in the first component and with finite intensity, its intensity measure can be disintegrated as
$$\bbE\Psi(B\times U)=\ova(Z)\lambda_n(B){\mathcal R}^*(U)$$
with a Borel probability measure ${\mathcal R}^*$ on $S^{n-1}$. If $\ova(Z)>0$ then $\cR^*$ is uniquely determined and it is called {\it oriented rose of directions} of $Z$ (cf.\ \cite{Rataj:15}). 
Note that this notion is in general different
from the usual \emph{oriented rose of directions} $\cR$, which is
defined under regularity conditions on $Z$ such that there is an outer
normal at $\cH^{n-1}$-almost all points in $\partial Z$. 
Both notions coincide if $\cH^{n-1}(\partial Z\setminus \partial^*
Z)=0$, for instance when $Z$ is a topologically regular
element of the extended convex ring, like in the case of a Boolean
model $Z$ of full-dimensional convex particles. 

We are now ready to prove our second main result.

\begin{proof}[Proof of Theorem~\ref{thm:contact}] 
If $\overline p=1$ then $Z=\R^n$ almost surely,
$\ova(Z)=0$, and \eqref{eq:contactDistStat} holds. For
$\overline p<1$  observe that 
\begin{align*}
(1-\overline p)H'_Q(0+)
 =\lim_{t\to \infty}(t^{n}\kappa_n)^{-1}  \lim_{r\to 0+}r^{-1}, 
\bbE\lambda_n(M_{r,t})
\end{align*}
with the set $M_{r,t}=[(Z\oplus(-rQ))\setminus Z]\cap B(0,t)$. 
We may assume $Q\subset B(0,1)$, and abbreviate   
$Z_s=Z\cap B(0,s)$, $s\ge 0$. For $t>1$,  $r\in (0,1)$ and $R_{r,t}$ being the annulus $B(0,t-1+r)\setminus B(0,t-1)$, we have  
\[
[Z_{t-1}\oplus  (-rQ)]\setminus Z_{t-1}\subset M_{r,t}\cup R_{r,t}
\]
and 
\[
M_{r,t}\subset [Z_{t+1}\oplus
(-rQ)]\setminus Z_{t+1}
\]
Due to 
\[
\lim_{t\to \infty}(t^{n}\kappa_n)^{-1}  \lim_{r\to 0+}r^{-1} 
\lambda_n(R_{r,t})=0, 
\] 
$\lim_{t\to \infty} t^n/(t\pm1)^n=1$, and $\lambda_n\left([Z_{t}\oplus  (-rQ)]\setminus Z_{t}\right)= G(-rQ,\1_{Z_{t}})$
we have
\begin{align}\label{LL}
(1-\overline p)H'_Q(0+)=
\lim_{t\to \infty}(t^{n}\kappa_n)^{-1}  \lim_{r\to 0+}r^{-1} 
\bbE G(-rQ,\1_{Z_{t}}). 
\end{align}
Assume that  $\ova(Z)<\infty$. 
Then \eqref{eqDefG}, Lemma \ref{l-aux}.(i) and Lemma \ref{lem:VqV}.(c) imply
\[
r^{-1}G(-rQ,\1_{Z_{t}})\le \sum_{0\not=u\in Q} V^{\{u\}}(\1_{Z_{t}})\le  (\# Q)  V(\1_{Z_{t}}),
\] which  gives the uniformly integrable
upper bound $(\# Q) P(Z\cap B(0,t))$. This allows us to use Lebesgue's dominated convergence
theorem for the limit $r\to 0+$ {when $t$ is fixed}. 
Hence,  Proposition \ref{prop:finiteQ} gives 
\begin{align}\label{ineq:rLimit}
\lim_{r\to 0+}r^{-1} 
\bbE G(-rQ,\1_{Z_{t}}) =\bbE V^Q(\1_{Z_{t}}).
\end{align}
As 
\[
(\partial^*Z)\cap \INt B(0,t)\subset \partial^*Z_t\subset 
[(\partial^*Z)\cap \INt B(0,t)]\cup tS^{n-1},
\]
$\lim_{t\to \infty}(t^{n}\kappa_n)^{-1}\cH^{n-1}(tS^{n-1})=0$,  $0\le h(-Q,\cdot)^+\le1$, the definition of $V^Q(\cdot)$ and \eqref{obsolete} 
yield
\begin{align*}\lim_{t\to \infty}&\,(t^{n} \kappa_n)^{-1} \bbE
V^Q(\1_{Z_{t}})\\&= 
\lim_{t\to \infty}(t^{n} \kappa_n)^{-1} \bbE
\int_{(\partial^*Z)\cap \INt B(0,t)} h(Q,\Delta_{\1_{Z_t}}(x))^+ \,\cH^{n-1}(dx).
\end{align*}
As $\Delta_{1_Z}$ is locally  defined according to \cite[p. 154]{AFP00}, we have $\Delta_{\1_{Z_t}}(x)=\Delta_{\1_{Z}}(x)$ for $\cH^{n-1}$-almost every $x\in (\partial^*Z)\cap \INt B(0,t)$, so 
\begin{align*}\lim_{t\to \infty}\,(t^{n} \kappa_n)^{-1} \bbE
V^Q(\1_{Z_{t}})
&=\lim_{t\to \infty}(t^{n} \kappa_n)^{-1} \bbE   
\int_{\INt B(0,t)\times S^{n-1}} h(-Q,v)^+ \,\Psi(d(x,v))
\\&=\ova(Z)\int_{S^{n-1}} h(-Q,v)^+{\cR}^*(dv). 
\end{align*}
The combination of this with \eqref{ineq:rLimit} and \eqref{LL} completes the proof in the case  $\ova(Z)<\infty$.  

Consider the  case where  $\ova(Z)=\infty$.
Approximating 
$\1_{Z_{t}}$ by mollifications $f_j\in C^1_c$ with non-negative
$\rho$, inequality \eqref{GG},  Proposition  \ref{PG} and Lemma \ref{lem:VqV}.(b) give  
\begin{align*}
\liminf_{r\to 0+}r^{-1} G(-rQ,\1_{Z_{t}})\ge 
\liminf_{r\to 0+}r^{-1} G(-rQ,f_j)\ge  V^{Q}(f_j)\ge     s V(f_j),  
\end{align*}
where $s>0$ is the inradius of $\conv(Q\cup\{0\})$; note that the latter set has interior points by assumption.  Proposition \ref{propDirVar}.(c) now implies  
\begin{align*}\liminf_{r\to 0+}
r^{-1} G(-rQ,\1_{Z_{t}})\ge  s V(\1_{Z_{t}}), 
\end{align*}
and insertion into  \eqref{LL} and using Lebesgue's dominated convergence theorem gives 
\begin{align}\label{eq:new}
(1-\overline p)H_Q'(0+)\ge\lim_{t\to\infty} s\frac{\bbE P(Z\cap B(0,t))}{t^n\kappa_n}=s\overline P(Z)=\infty, 
\end{align}
due to Lemma \ref{lem:W}.  \smallskip

Now let $Z$ be isotropic. If $\ova(Z)=0$, the claim is trivial. If  $0<\ova(Z)<\infty$, the measure ${\cR}^*$ is the uniform distribution on $S^{n-1}$ and the definition of the mean width gives the required
relation. If $\ova(Z)=\infty$, equation \eqref{eq:isotrop} holds for
$Q=\{0\}$, so we may assume that there is an $u_0\in  S^{n-1}$ and a number  $s>0$ such that $su_0\in Q$. Then \eqref{LL}, $G(-rQ,\cdot)\ge s G(r\{-u_0\},\cdot)$, Proposition \ref{prop:finiteQ} and Lemma \ref{prop:VQ}.(a) yield 
\begin{align*}\label{LL}
(1-\overline p)H'_Q(0+)\ge \frac s2\lim_{t\to\infty}(t^{n}\kappa_n)^{-1}  
\bbE V_{u_0}(\1_{Z_{t}}). 
\end{align*}
As $Z$ is isotropic, $\bbE V_{u_0}(\1_{Z_{t}})=\bbE
V_{u}(\1_{Z_{t}})$ for all $u\in S^{n-1}$, and
\eqref{eq:rotatMittel} gives 
\[
(1-\overline p)H'_Q(0+)\ge \lim_{t\to\infty}
s(2n\kappa_n)^{-1}\int_{S^{n-1}} 
\frac{\bbE
	V_{u}(\1_{Z_{t-1}})}{t^n\kappa_n}\cH^{n-1}(du)=
\frac{s\kappa_{n-1}}{n\kappa_n} \ova(Z)=\infty.
\]
Thus, assertion  \eqref{eq:isotrop} is shown and the proof is complete. 
\end{proof}

Note that the only  assumption on the random set $Z$ 
in Theorem \ref{thm:contact} is stationarity. 
The use of the bounded variation concept allows us
to avoid any kind of integrability condition, which is usually present
in similar results. For instance, \eqref{eq:contactDistStat} 
was shown in \cite{KR06} for ``gentle''
random sets and compact $Q$. A variant of \eqref{eq:contactDistStat} for
non-stationary $Z$, where 
$H_{Q}(\cdot)$ also depends on the position of (the compact, convex
set)  $Q$ and on the
outer normal of the contact
point,  was shown in  \cite{HL:00} 
for  a grain model with compact convex grains. A related result is
given in \cite[Theorem 4.1]{Vil:10}, where 
the derivative of the spherical contact distribution function 
of certain non-stationary  Boolean models $Z$  is
determined for $0\le r\le R$, where $R$ is the reach of the typical
grain of $Z$. Under appropriate assumptions, even the (right sided)
second derivative at zero is given there. 
All three named papers rely on the (local) finiteness of certain measures
associated to $Z$. 
The price to pay for the generality of Theorem \ref{thm:contact} 
are the severe restrictions on
the structuring element $Q$. However, 
\eqref{eq:contactDistStat} cannot hold for general compact $Q$, 
as the example of a stationary hyperplane process together with 
$Q=B(0,1)$ shows.

\end{document}